\documentclass[11pt]{amsart}

\usepackage{amsmath, a4wide}
\usepackage{amssymb,amsthm,amsfonts}
\usepackage{amsrefs}
\usepackage[utf8]{inputenc}
\usepackage[T1]{fontenc}
\usepackage{xcolor}

\renewcommand{\mathcal}{\mathscr}

\def\R {\mathbb{R}}

\def\Z {\mathbb{Z}}

\def\Xint#1{\mathchoice
{\XXint\displaystyle\textstyle{#1}}%
{\XXint\textstyle\scriptstyle{#1}}%
{\XXint\scriptstyle\scriptscriptstyle{#1}}%
{\XXint\scriptscriptstyle\scriptscriptstyle{#1}}%
\!\int}
\def\XXint#1#2#3{{\setbox0=\hbox{$#1{#2#3}{\int}$ }
\vcenter{\hbox{$#2#3$ }}\kern-.6\wd0}}

\def\dashint{\Xint-}

\renewcommand{\epsilon}{\varepsilon}
\newcommand{\eps}{\varepsilon}

\renewcommand{\leq}{\leqslant}
\renewcommand{\le}{\leqslant}
\renewcommand{\geq}{\geqslant}

\renewcommand{\div}{\mathrm{div}}

\newcommand{\Per}{\mathrm{Per}}

\newtheorem{proposition}{Proposition}[section]
\newtheorem{assumption}{Assumption}[section]
\newtheorem{theorem}[proposition]{Theorem}
\newtheorem{corollary}[proposition]{Corollary}
\newtheorem{lemma}[proposition]{Lemma}
\theoremstyle{definition}
\newtheorem{definition}[proposition]{Definition}

\numberwithin{equation}{section}
\title{Stability of the ball under volume preserving fractional mean curvature flow }

\author[A. Cesaroni, M. Novaga]{}

\keywords{Fractional mean curvature flow, nearly spherical sets,  long time behavior, Alexandrov theorem}

\email{annalisa.cesaroni@unipd.it}
\email{matteo.novaga@unipi.it}
 
\thanks{The authors were supported by the INDAM-GNAMPA and by the PRIN Project 2019/24 {\it Variational methods for stationary and evolution problems with singularities and interfaces}.}

\begin{document}
\maketitle

\centerline{\scshape Annalisa Cesaroni }
{\footnotesize
 \centerline{Department of Statistical Sciences}
   \centerline{University of Padova}
   \centerline{Via Cesare Battisti 141, 35121 Padova, Italy  }
} 
\medskip

\centerline{\scshape Matteo Novaga}
{\footnotesize
\centerline{ Department of Mathematics}
   \centerline{University of Pisa}
   \centerline{Largo Bruno Pontecorvo 5, 56127 Pisa, Italy  }
}

\begin{abstract}
We consider the volume constrained fractional mean curvature flow of a nearly spherical set, 
and prove long time existence and asymptotic convergence to a ball. 
The result applies in particular to convex initial data, under the assumption of global existence. 
Similarly, we show exponential convergence to a constant 
for the fractional mean curvature flow of a periodic graph.
\end{abstract}

\tableofcontents

\section{Introduction}

We recall the definition of  fractional perimeter and fractional mean curvature of a set, as introduced by Caffarelli, Roquejoffre and Savin in \cite{crs}. 
Let $s\in (0,1)$; for a set $A\subseteq\R^n$,  with $C^{1,1}$ boundary, we let
\begin{eqnarray*} \Per_s(A)&:=&\int_A\int_{\R^n\setminus A} \frac{1}{|x-y|^{n+s}}dydx\label{ps}=\frac{1}{s}\int_{A}\int_{\partial A}\frac{(y-x)\cdot \nu(y)}{|x-y|^{n+s}} dH^{n-1}(y)  dx
\\   H_A^s(x)&:=&\int_{\R^n\setminus A} \frac{1}{|x-y|^{n+s}}dy-\int_{A} \frac{1}{|x-y|^{n+s}}dy\label{hs}=\frac{2}{s } \int_{\partial A}\frac{(y-x)\cdot \nu(y)}{|x-y|^{n+s}} dH^{n-1}(y),
\end{eqnarray*} 
where the integrals are intended in the principal value sense, and $\nu(x)$ denotes the exterior normal to $A$ at $x\in \partial A$. 

The fractional mean curvature flow starting from an initial set 
$E_0\subseteq \R^n$ is a family of sets  $E_t$, parametrized by $t\geq 0$ and  defined by  the geometric evolution law   
\begin{equation}\label{mcfbis}
\partial_t x_t\cdot \nu(x_t)=- H^s_{E_t}(x_t) \qquad x_t\in \partial E_t.
\end{equation} 
 
Similarly, the volume preserving fractional mean curvature flow is defined by the equation
\begin{equation}\label{mcf}
\partial_t x_t\cdot \nu(x_t)=-H^s_{E_t}(x_t)+\overline{H^s_{E_t}}\ \qquad x_t\in \partial E_t,
\end{equation} 
where $\overline{H^s_{E_t}}$ is the average fractional mean curvature, defined as
\begin{equation}\label{media} \overline{H^s_{E_t}}= \dashint_{\partial E_t} H^s_{E_t}(y)dH^{n-1}(y).  
\end{equation} 
If $E_t$ is a smooth solution to \eqref{mcf}, then the volume of $E_t$ is constant in time, 
and the fractional perimeter $P_s(E_t)$ is strictly decreasing unless $E_t$ is a ball.

A short time existence result for smooth solutions of both \eqref{mcfbis} and \eqref{mcf} starting from compact $C^{1,1}$ initial sets was recently provided in  \cite{jlm}. 
On the other hand, existence of weak solutions  has been obtained by different authors, see  \cites{cs, cmp, i}.
In \cite{cnr} the authors showed that the flow \eqref{mcfbis} is convexity preserving, also in the presence of a time dependent forcing term. 

Concerning the long time behavior of solutions,
in \cite{cinti}  it has been proved that smooth convex solutions to \eqref{mcf} converge to a ball, up to suitable translations possibly depending on time. In \cite{cn} the authors discuss the long time behavior of entire Lipschitz graphs evolving under \eqref{mcfbis}, showing that asymptotically flat graphs and periodic graphs converge to hyperplanes uniformly in $C^1$ . 

\smallskip

In this paper we shall mainly consider the long time behavior of the volume preserving flow \eqref{mcf}, 
with initial data $E_0$ which are nearly spherical, according to the following definition: 
\begin{definition}[Nearly spherical set] \label{defi1} 
Let $B_m\subset \R^n$ the $n$-dimensional ball centered at $0$ and  with volume $m>0$. 
A nearly spherical set $E$ is defined as follows:  
\begin{equation}\label{set} E:=\{rx, x\in \partial B_m, r\in [0, 1+u(x)]\}
\end{equation} where   $u:\partial B_m\to \R$ is a $C^{1,1}$ function
with $\|u\|_{C^1} <1$.
(In the sequel we may extend the function $u$ to $\R^n\setminus \{0\}$ by letting $u(x):=u(x/|x|)$.)
\end{definition} 

In particular, in Theorem \ref{convergenza1} we show that, if the initial set is nearly spherical with $\eps$ sufficiently small, then the volume preserving flow exists smooth for all times and converges exponentially fast in $C^\infty$ to a  translate of the reference ball.  This result  provides  an improvement of the result in \cite{cinti} discussed above, see Corollary \ref{corocinti}, ruling out the possibility of indefinite translations and giving the exponential rate of convergence.  Similar results in the local setting date back to \cite{h} for the case of convex initial sets, to \cites{at, e} for nearly spherical initial sets, and more recently to \cite{jmps} in dimension $2$ for weak solutions   starting from a general bounded set of finite perimeter.

The main technical tool used in the proof is a quantitative Alexandrov type estimate for nearly spherical sets, proved in Theorem  \ref{mediacurv}. 
In \cites{cabre, cfmn} a fractional analog of the classical Alexandrov theorem was established, namely that the boundary of a bounded smooth set with constant fractional mean curvature is necessarily a sphere. 
More generally, in \cite{cfmn} it is proved that the Lipschitz constant of the fractional mean curvature of  a set with smooth  boundary controls linearly  its $C^2$-distance from a single sphere.  On the other hand, the Alexandrov type estimate \eqref{deficit2} provides a linear control on the $H^{\frac{1+s}{2}}$ distance from the reference sphere of a nearly spherical set, in terms of the $L^2$ deficit of the fractional curvature, so giving  a  stability results for nearly spherical sets.

The inequality \eqref{losia2} in  Theorem \ref{mediacurv} can be interpreted as a \L ojasiewicz-Simon inequality (see \cite{c}) for the energy functional $\Per_s(E)$.  Indeed, it  bounds the
difference in energy between  the ball, which is a critical point of the fractional perimeter, and a nearly spherical set in terms of $L^2$ norm
of the first variation of the energy, that is, the  the $L^2$ deficit of the fractional mean curvature.  
 The idea of using these inequalities for proving convergence of solutions to parabolic equations  goes back to the seminal paper of Simon \cite{s}, and has been used for geometric flows in different contexts,  see for instance \cite{mp} and references therein. 

After this work was completed, we were informed that an Alexandrov type inequality similar to the one  in Theorem  \ref{mediacurv} 
has been independently established in \cite{ku}, where the authors prove the existence of flat flows for the fractional volume preserving mean curvature
flow, and characterize the long time behaviour of its discrete-in-time approximation in low dimension (since 
in higher dimension
it is missing an Alexandrov type theorem for non smooth sets in the fractional case). In particular, they show that the discrete flow starting from a bounded set of finite fractional perimeter converges exponentially fast to a single ball.
  
Finally, we observe that the case of entire  periodic  graphs evolving by \eqref{mcfbis} presents some analogies with the flow of nearly spherical sets by the volume preserving mean curvature flow. In particular, in the last section of the paper we establish a similar \L ojasiewicz-Simon inequality for periodic graphs, which allows us to
improve the long time convergence to hyperplanes proved in \cite{cn},
getting exponential convergence in $C^\infty$, see Corollary \ref{expografo}.

\smallskip

\paragraph{\bf Acknowledgments}
The authors are members of INDAM-GNAMPA. The second author was supported by the PRIN Project 2019/24.

\section{An Alexandrov type estimate for nearly spherical sets}

In this section we show that if $E$ is a nearly spherical set, the $H^{\frac{1+s}{2}}$-distance  
  of $\partial E$  to the reference  sphere is linearly  bounded  in terms of the $L^2$-deficit of $H^s_E$ with respect to 
  its average, whenever $\partial E$ is a sufficiently small perturbation
of the  sphere.  The analogous result of Theorem \ref{mediacurv}  in the case of the classical mean curvature has  been proved in \cite[Theorem 1.10]{mk} and  \cite[Theorem 1.2]{mps}.   

First of all we observe that, via a simple rescaling argument, we may reduce to the case of volume $1$. Indeed, 
  it is clear that  the results  we are going to prove can also be stated for sets parametrized  over a ball $B_m$ with volume $m$, with constants  depending on $m$. 
The dependence of such constants on $m$ can be made  uniform as $m$ varies on bounded intervals.  
 
We introduce  the (squared) fractional Gagliardo seminorm of $u$:
\begin{equation}\label{gs}[u]^2_{\frac{1+s}{2}}:=\int_{\partial B} \int_{\partial B} \frac{(u(x)-u(y))^2}{|x-y|^{n+s}}  dH^{n-1}(y)  dH^{n-1}(x).  \end{equation}
Moreover we will indicate with $\|u\|_2^2$ the squared $L^2(\partial B)$ norm of $u$, that is $\int_{\partial B} u^2(x)dH^{n-1}(x)$. 
We will endow  the fractional Sobolev space $H^{\frac{1+s}{2}}(\partial B)$ with the norm given by the square root  of  
\[\|u\|^2_{L^2(\partial B)}+  [u]^2_{\frac{1+s}{2}}. \]
We introduce  the hypersingular Riesz operator on the sphere, which is defined (up to constants depending on $s$ and $n$) as 
\begin{equation}\label{lapl}
(-\Delta)^{\frac{1+s}{2}}u(x)=2 \int_{\partial B} \frac{u(x)-u(y)}{|x-y|^{n+s}}dH^{n-1}(y).
\end{equation} 
Note that
 \begin{equation}\label{fract} [u]^2_{\frac{1+s}{2}} = \int_{\partial B} u(x)(-\Delta)^{\frac{1+s}{2}}u(x)dH^{n-1}(x).  \end{equation}

We recall the following results on the asymptotics of these norms  (see \cite{bbm}, \cite{ms}) and of the curvatures  (\cite{cdnp}), as $s\to 0, 1$.
\begin{theorem}
Let $u\in H^1(\partial B)$.  There exist dimensional constants $c(n),k(n)>0$ such that $\lim_{s\to 1^-} (1-s)\|u\|_{\frac{1+s}{2}}^2=c(n)\|\nabla u\|_2^2$ and $\lim_{s\to 0^+} s\|u\|_{\frac{1+s}{2}}^2=k(n)\|u\|_2^2$. 
Let $E\subseteq\R^n$ be a bounded set with  $C^{1,1}$ boundary. Then there exist dimensional constants $c(n),k(n)>0$ such that 
$\lim_{s\to 1^-} (1-s) H_E^s(x)=c(n)H_E(x)$, uniformly in $x\in\partial E$, where $H_E(x)$ is the classical mean curvature,
 and $\lim_{s\to 0^+} sH_E^s(x)=k(n)|E|$ uniformly for $x\in\partial E$. 
\end{theorem} 

We now  state the main result of this section. For the case in which $s=1$,  we refer to  \cite[Theorem 1.10]{mk} and  \cite[Theorem 1.2]{mps}.

\begin{theorem}\label{mediacurv} 
Assume that $E$ is a nearly spherical set such that $|E|=|B|$ and  the barycenter of $E$ is the same as $B$, that is, $\int_{E}x dx=0$.
Then there exist positive  constants $C(n,s)>0$ and $\eps_0(n,s)\in (0,1)$  such that if $ \|u\|_{C^1}< \eps_0$ there holds 
\begin{equation}\label{deficit2} [u]_{\frac{1+s}{2}}^2+\|u\|_2^2 \leq C(n,s) \|H_E^s-\overline{H_E^s} \|_{L^2(\partial E)}^ 2. \end{equation}
Moreover  there exists a positive constant $K(n,s)>0$ depending on $n,s$ such  that 
\begin{equation}\label{losia2} \Per_s(E)-\Per_s(B)\leq K(n,s) \| H_E^s-\overline{H_E^s} \|_{L^2(\partial E)}^ 2. \end{equation} 
\end{theorem} 
First of all we observe that it is sufficient to prove \eqref{deficit2}, since \eqref{losia2} is a consequence of \eqref{deficit2} and of the rigidity  inequality 
\begin{equation}\label{dino}  \Per_s(E)-\Per_s(B)\leq c(n)  [u]_{\frac{1+s}{2}}^2\end{equation} 
which was proved in \cite[Theorem 6.2]{dinoruva}. 

In order to show \eqref{deficit2}   we   need some preliminary computations. 
We first compute the fractional mean curvature of $E$ in spherical coordinates. We fix a point $\bar x$ on $\partial E$. Then $\bar x=(1+u(x))x$ for some $x\in\partial B$. 
We rewrite the curvature $H^s_E(\bar x)$ as defined in \eqref{hs} as an integral over $\partial B$, by using the area formula. 
  Observe that if $\bar y=(1+u(y))y\in \partial E$ with $y\in \partial B$, then 
$\nu(\bar y)=\frac{(1+u(y))y-\nabla u(y)}{\sqrt{ |1+u(y)|^2+|\nabla u(y)|^2}}$ where $\nabla u$ is the tangential gradient of $u$. Moreover
the tangential jacobian   (see e.g.  \cite[Lemma 4.1]{mk}) is given at a point $\bar y=(1+u(y))y$ by 
$\sqrt{ |1+u(y)|^2+|\nabla u(y)|^2} (1+u(y))^{n-2}. $

So by the area formula  the curvature  $H^s_E(\bar x)$ coincides with 
\begin{eqnarray} \nonumber &&  H_E^s(x(1+u(x))\\ 
\nonumber &=& \frac{2}{s} \int_{\partial B} \frac{(y-x+yu(y)-xu(x))\cdot  [(1+u(y))y-\nabla u(y)]}{|y-x+yu(y)-xu(x)|^{n+s}} (1+u(y))^{n-2} dH^{n-1}(y)\\ 
\label{uno} &=& \frac{2}{s} \int_{\partial B} \frac{ u(y)-u(x)}{|y-x+yu(y)-xu(x)|^{n+s}} (1+u(y))^{n-1} dH^{n-1}(y)\\
\label{due} &&+ \frac{1}{s} \int_{\partial B} \frac{ (u(x)+1) |x-y|^2}{|y-x+yu(y)-xu(x)|^{n+s}} (1+u(y))^{n-1} dH^{n-1}(y)\\ 
\label{tre} &&+ \frac{2}{s} \int_{\partial B} \frac{  (1+u(x))(x-y)\cdot \nabla u(y)}{|y-x+yu(y)-xu(x)|^{n+s}} (1+u(y))^{n-2} dH^{n-1}(y),
\end{eqnarray}
where we used that  $ \nabla u(y)\cdot y=0$, $(y-x)\cdot y=\frac{|y-x|^2}{2}=1-x\cdot y$ .

We start  with some integral estimates of the difference between the curvature of $E$ (expressed  in spherical coordinates) and the curvature of the ball.  
By using the definition \eqref{hs} and the fact that $(y-x)\cdot y=\frac{|y-x|^2}{2}$ for all $x,y\in\partial B$, we get that the curvature of the ball $H^s_B$ is constant and coincides with
\begin{equation}\label{hsB} H_B^s=\frac{1}{s} \int_{\partial B}\frac{	1}{|x-y|^{n+s-2}} dH^{n-1}(y).\end{equation}  
 
\begin{lemma}\label{lemmacurv} 
Let $E$ be a nearly spherical set as in Definition \ref{defi1}. 
For every $x\in \partial B$, we denote with $H^s_E(x(1+u(x)))$ the fractional mean curvature of $E$ at the point $x(1+u(x))\in \partial E$.
Then there holds  \begin{eqnarray} \nonumber 
\int_{\partial B}( H^s_E(x(1+u(x))-H^s_B) &=& -\frac{n+s}{2} [u]^2_{\frac{1+s}{2}} (1+O(\|u\|_{C^1}) - sH_B^s \int_{\partial B}  u(x)     \\&& + \frac{s(s+1) }{2} H_B^s\|u\|_2^2 (1+O(\|u\|_{C^1}),\label{primoint} \\
  \int_{\partial B}u(x)( H^s_E(x(1+u(x))-H^s_B) &= & \left([u]^2_{\frac{1+s}{2}} -s H_B^s \|u\|_2^2\right)  (1+ O(\|u\|_{C^1})),\label{secondoint}
\end{eqnarray} 
where $O(\|u\|_{C^1})$ denotes a  function $f(x)$ such that $|f(x)|\leq C\|u\|_{C^1}$ for all $x\in \partial B$, for some $C>0$. 

Moreover, if $|E|=|B|$ we may rewrite \eqref{primoint} as
 \begin{equation} 
\int_{\partial B}( H^s_E(x(1+u(x)))-H^s_B) = -\frac{n+s}{2} \left([u]^2_{\frac{1+s}{2}} -s H_B^s \|u\|_2^2\right)  (1+ O(\|u\|_{C^1})).\label{terzoint}\end{equation} 

\end{lemma} 
\begin{proof}  
We first notice that
\begin{eqnarray*} (1+u(y))^{n-1}&=& 1+(n-1)u(y)+\frac{(n-1)(n-2)}{2} u^2(y)+ O(\|u\|_{C^1}^3)\\
\nonumber   \frac{1}{|y-x+yu(y)-xu(x)|^{n+s}} &=& \frac{1}{|y-x|^{n+s}} -\frac{n+s}{2|x-y|^{n+s}} (u(x)+u(y)) +\\ 
&+& \frac{(n+s)(n+s+2)}{8|x-y|^{n+s}}(u^2(y)+u^2(x))+\frac{(n+s)^2}{4|x-y|^{n+s}}u(x)u(y)\\&-& \frac{n+s}{2|x-y|^{n+s}}\frac{|u(y)-u(x)|^2 }{|x-y|^2}+ \frac{1}{|x-y|^{n+s}}O(\|u\|_{C^1}^3),
\end{eqnarray*}  
Putting together the previous expansions we conclude that 
\begin{eqnarray} 
\nonumber  \frac{(1+u(y))^{n-1}}{|y-x+yu(y)-xu(x)|^{n+s}}  &=&
 \frac{1}{|x-y|^{n+s}} - \frac{n+s}{2|x-y|^{n+s}}\frac{|u(y)-u(x)|^2 }{|x-y|^2} \\ \nonumber &+ &  \frac{n-1}{|x-y|^{n+s}}u(y) - \frac{n+s}{2|x-y|^{n+s}} (u(x)+u(y)) \\ \nonumber &+& \frac{(n-1)(n-2)}{2|x-y|^{n+s}}  u^2(y)  -\frac{(n-1)(n+s)}{2|x-y|^{n+s}}   u(y) (u(x)+u(y)) \\ \nonumber  &+&    \frac{(n+s)(n+s+2)}{8|x-y|^{n+s}}(u^2(y)+u^2(x))+\frac{(n+s)^2}{4|x-y|^{n+s}}u(x)u(y)  \\ &+& \frac{1}{|x-y|^{n+s}}O(\|u\|_{C^1}^3).  \label{tay4} 
\end{eqnarray} 

We shall compute $\int_{\partial B} H^s_E(x(1+u(x))dH^{n-1}(x)$ and $\int_{\partial B} u(x) H^s_E(x(1+u(x))dH^{n-1}(x)$ by considering separately the three terms in
 \eqref{uno}, \eqref{due}, \eqref{tre}.  


\vspace{0.2cm} 

\noindent {\bf First term \eqref{uno}}.
We  use the   Taylor expansion \eqref{tay4} in   the term in \eqref{uno} and we integrate it on $\partial B$ with respect to $x$:  We  get 
\begin{multline}\label{uno1int} \frac{2(n-1)}{s} \int_{\partial B} \int_{\partial B}  \frac{ (u(y)-u(x))u(y)}{|y-x|^{n+s}} \left[1+  O(\|u\|_{C^1}) \right]dH^{n-1}(y)\\= \frac{(n-1)}{s} \int_{\partial B} \int_{\partial B}  \frac{ (u(y)-u(x))^2}{|y-x|^{n+s}} \left[1+  O(\|u\|_{C^1})) \right]dH^{n-1}(y)dH^{n-1}(x)= \frac{n-1}{s} [u]^2_{\frac{1+s}{2}}(1+O(\|u\|_{C^1})).\end{multline}
We multiply now the term in \eqref{uno} by $u(x)$ and integrate, also using the Taylor expansion and we get
\begin{equation}\label{uno1int1} \frac{2}{s}  \int_{\partial B} \int_{\partial B}  \frac{ (u(y)-u(x))u(x)}{|y-x|^{n+s}} \left[1+ O(\|u\|_{C^1}) \right]dH^{n-1}(y)= - \frac{1}{s} [u]^2_{\frac{1+s}{2}}(1+O(\|u\|_{C^1})). \end{equation}


\vspace{0.2cm} 

\noindent {\bf Second term   \eqref{due}}. We  use the   Taylor expansion \eqref{tay4} in   the term in \eqref{due} and we integrate it on $\partial B$ with respect to $x$ recalling \eqref{hsB}: We get 
\begin{eqnarray}\label{due1intprimo} & &
\int_{\partial B} H^s_B dH^{n-1}(x) -\frac{n+s}{2s}  [u]^2_{\frac{1+s}{2}}(1+O(\|u\|_{C^1})) \\ \nonumber &-&s H_B^s \int_{\partial B}  u(x)  dH^{n-1}(x)  (1O(\|\nabla u\|_{C^0}^2)
 \\ \nonumber  & +&\frac{n^2-4n+s^2+2s+4}{4} H_B^s\int_{\partial B} u^2(x)dH^{n-1}(x) (1+O(\|u\|_{C^1}))\\&+&\frac{s^2-n^2-4+4n}{4s} \int_{\partial B} \int_{\partial B} \frac{u(x)u(y)}{|x-y|^{n+s-2} } dH^{n-1}(x)dH^{n-1}(y) (1+O(\|u\|_{C^1})).\nonumber \end{eqnarray}

We multiply the term  \eqref{due} by $u(x)$ and integrate and we get, recalling \eqref{hsB}, 
\begin{eqnarray}\nonumber  H_B^s   \int_{\partial B}  u(x)  dH^{n-1}(x) &+&   [u]^2_{\frac{1+s}{2}}O(\|u\|_{C^1})+\frac{2-n-s}{2}H_B^s \int_{\partial B}  u^2(x)  (1+O(\|u\|_{C^1})) dH^{n-1}(x) \\ &+& \frac{n-2-s}{2s}\int_{\partial B}\int_{\partial B} \frac{u(x)u(y)}{|x-y|^{n+s-2}}dH^{n-1}(y)dH^{n-1}(x).\label{due1int1} \end{eqnarray}


\vspace{0.2cm} 

\noindent {\bf Third  term \eqref{tre}}.
Integrating \eqref{tre} on $\partial B$ with respect to $x$ and using the Taylor expansion \eqref{tay4} (with $n-1$ in place of $n$) we get
\begin{eqnarray}\label{tre1int} && \frac{2}{s} \left[1-\frac{n+s}{2}\right]\int_{\partial B}\int_{\partial B}    \frac{u(x) (x-y)\cdot \nabla u(y)}{|x-y|^{n+s}}   \left(1 + O(\|u\|_{C^1})\right) dH^{n-1}(y)dH^{n-1}(x)\\ \nonumber && -\frac{n+s}{s}\int_{\partial B}\int_{\partial B}    \frac{ (x-y)\cdot \nabla u(y) }{|x-y|^{n+s}}\frac{|u(y)-u(x)|^2}{|x-y|^2} (1+O(\|u\|_{C^1}))dH^{n-1}(y)dH^{n-1}(x) \nonumber.\end{eqnarray} 
In order to rewrite the two terms in \eqref{tre1int}, we are going to use the divergence theorem on $\partial B$. 
Let us fix $x\in \partial B$ and consider the map  $T(y)= \frac{(u(y)-u(x))(y-x)}{|x-y|^{n+s}}$ for $y\in \partial B$. We compute the Jacobian of $T$:
\[ JT(y)= \frac{u(y)-u(x)}{|x-y|^{n+s}}\left[\delta_{ij}-(n+s)\frac{y-x}{|y-x|}\otimes \frac{y-x}{|y-x|}\right]+\frac{\nabla u(y)\otimes(y-x)}{|x-y|^{n+s}}\]
so that 
\begin{equation}\label{diver} \div T(y)= -\frac{s(u(y)-u(x))}{|x-y|^{n+s}}+\frac{\nabla u(y)\cdot (y-x)}{|x-y|^{n+s}}.\end{equation} 
Recalling that $\nabla u(y)\cdot y=0$, and that $y\cdot (y-x)=\frac{|y-x|^2}{2}$ we have that 
\[y JT(y)\cdot y= \frac{u(y)-u(x)}{|x-y|^{n+s}}-\frac{n+s}{4} \frac{u(y)-u(x)}{|x-y|^{n+s-2}}\]
and so we get that the tangential divergence of $T$ is given by 
\[ \div^\tau T(y)=-(s+1)\frac{u(y)-u(x)}{|x-y|^{n+s}}+\frac{(u(y)-u(x))(n+s)}{4|x-y|^{n+s-2}}+
\frac{\nabla u(y)\cdot (y-x)}{|x-y|^{n+s}}.\]
By the divergence theorem on $\partial B$, we have that (recalling that the curvature of $B$ is $n-1$) 
\[ -(s+1) \int_{\partial B}  \frac{u(y)-u(x)}{|x-y|^{n+s}}+\frac{n+s}{4} \int_{\partial B} \frac{(u(y)-u(x))}{|x-y|^{n+s-2}}+\int_{\partial B} 
\frac{\nabla u(y)\cdot (y-x)}{|x-y|^{n+s}}= \frac{n-1}{2}  \int_{\partial B}  \frac{u(y)-u(x)}{|x-y|^{n+s-2}}. \]

Multiplying by $u(x)$ and integrating on the sphere we get \begin{multline}\label{div1}
\int_{\partial B}\int_{\partial B}    \frac{u(x) (x-y)\cdot \nabla u(y)}{|x-y|^{n+s}}= \frac{s+1}{2} [u]_{\frac{1+s}{2}}^2+\frac{n-s-2}{4} sH^s_B  \int_{\partial B} u^2(x) \\+\frac{s+2-n}{4} \int_{\partial B}\int_{\partial B}    \frac{u(x)  u(y)}{|x-y|^{n+s-2}}. \end{multline} 
Repeating the same kind of computations for  $S(y)= \frac{(u(y)-u(x))^3(y-x)}{|x-y|^{n+s+2}}$ we get that 
\[ \div^\tau S(y)=-(s+3)\frac{(u(y)-u(x))^3}{|x-y|^{n+s+2}}+\frac{(u(y)-u(x))^3(n+s+2)}{4|x-y|^{n+s}}+
3\frac{|u(y)-u(x)|^2 \nabla u(y)\cdot (y-x)}{|x-y|^{n+s+2}}.\]
Now we use the  the divergence theorem on $\partial B$, to get 
\[\int_{\partial B} \div^\tau S(y)=\frac{n-1}{2}\int_{\partial B} \frac{(u(y)-u(x))^3}{|x-y|^{n+s}}\]
 and integrating  again on $\partial B$, we obtain 
\begin{equation}\label{div2}\int_{\partial B}\int_{\partial B}    \frac{ (x-y)\cdot \nabla u(y) }{|x-y|^{n+s}}\frac{|u(y)-u(x)|^2}{|x-y|^2}  dH^{n-1}(y)dH^{n-1}(x)=0.\end{equation}

We use \eqref{div1} and \eqref{div2} in \eqref{tre1int} and we get 
\begin{multline}\label{tre2int} \frac{(s+1)(2-n-s)}{2s} [u]^2_{\frac{1+s}{2}}(1+O(\|u\|_{C^1}))+\frac{s^2-(n-2)^2}{4} H_B^s \int_{\partial B} u^2(x)\\+ \frac{(n-2)^2-s^2}{4s} \int_{\partial B}\int_{\partial B}    \frac{u(x)  u(y) }{|x-y|^{n+s-2}}. \end{multline}   
If we multiply by $u(x)$ the term in \eqref{tre} and integrate, recalling the Taylor expansion and using   \eqref{div1},  we get 
\begin{eqnarray}\label{tre2int1} && \frac{s+1}{s} [u]^2_{\frac{1+s}{2}}(1+O(\|u\|_{C^1})) \\ \nonumber &+&\frac{n-s-2}{2} H_B^s\int_{\partial B} u^2(x)+\frac{s+2-n}{2s} \int_{\partial B}\int_{\partial B}    \frac{u(x)  u(x)}{|x-y|^{n+s-2}}.\end{eqnarray} 


By using \eqref{uno1int}, \eqref{due1intprimo} and \eqref{tre2int} we conclude \eqref{primoint}. 
By using \eqref{uno1int1}, \eqref{due1int1} and \eqref{tre2int1} we conclude \eqref{secondoint}. 

Finally   the volume condition reads   \[|B|=\int_E dx=\int_{\partial B} \frac{(1+u(y))^n}{n} dH^{n-1}(y).\] 
So, recalling that $n |B|=\int_{\partial B} dH^{n-1}(y)$, and performing a Taylor expansion we get
\begin{equation}\label{volume} \int_{\partial B} u(y) dH^{n-1}(y)=-\frac{n-1}{2} \int_{\partial B} u^2(y) (1+O(\|u\|_{C^1})) dH^{n-1}(y).\end{equation} 
If we substitute in \eqref{primoint}, we conclude \eqref{terzoint}.
 
\end{proof} 

\begin{proof}[Proof of Theorem \ref{mediacurv}] 
As discussed above it is sufficient to prove the validity of \eqref{deficit2}. 
The proof of this estimate is divided in two main steps. In the first step, using Lemma \ref{lemmacurv} and a  Poincar\`e type inequality, we prove that there exist $\eps_0=\eps_0(n,s)>0$ and $C(n,s)>0$ depending on $n ,s$ such that 
 if $ \|u\|_{C^1}< \eps_0$ there holds 
\begin{equation}\label{deficit} [u]_{\frac{1+s}{2}}^2+\|u\|_2^2 \leq C(n,s) \|H_E^s-H_B^s \|_{L^2(\partial B)}^ 2. \end{equation}
In the second step,   by a rescaling argument   and by the area formula we deduce  \eqref{deficit2} from \eqref{deficit}.  
 
Along the proof,  $C(n,s)$ will indicate a constant depending on $n,s$ which may change from line to line. 

\noindent
{\bf First Step: proof of \eqref{deficit}}.  

We follow  \cite[Section 2]{i5},  where it is provided  the fractional counterpart of the classical estimates of Fuglede on nearly spherical sets, see \cite{f}.  We introduce   the $L^2(\partial B)$ orthonormal basis $Y_k^i$ of spherical harmonics of degree $k=0,1, \dots$. 
Of course we have that  $Y_0=\frac{1}{\sqrt{n|B|}}$, $Y^i_1=\frac{x_i}{\sqrt{|B|}} $ for $i=1\dots, n$.

We will denote with $\lambda_k^s$ the $k$-order eigenvalue of  the operator \eqref{lapl}, so there holds that  $(-\Delta)^{\frac{1+s}{2}}Y_k^i=\lambda^s_k Y_k^i$. 
It is possible to show (we refer to \cite[Proposition 2.3]{i5} and references therein) that  $\lambda^s_k>\lambda^s_{k-1}$ for all $k\geq 1$ and that 
\begin{equation}\label{eig} \lambda_0^s=0,\quad \lambda_1^s=s H^s_B, \quad \lambda_2^s= \frac{2n}{n-s}\lambda_1^s\geq 2\lambda_1^s. \end{equation}  

We write $u$ as a Fourier serie  with respect to the spherical harmonics, up to degree $2$:
\[u(x)= a Y_0+\sum_{i=1}^n b_i\cdot Y_1^i +R(x)=\frac{1}{n|B|} \int_{\partial B} u(y)dH^{n-1}(y)+\frac{1}{|B|}   \int_{\partial B} u(y)y\cdot x dH^{n-1}(y)+R(x)\] where  $R$ is orthogonal to the harmonics of degree $0, 1$, that is 
$\int_{\partial B} R(y)dH^{n-1}(y)=0$ and $ \int_{\partial B} y_i R(y)dH^{n-1}(y)=0$ for all $i$. 
We compute  
\begin{equation}\label{2norm} \|u\|_2^2=\int_{\partial B} u^2(x)dH^{n-1}(x)=\frac{ 1}{|B|} \left( \int_{\partial B} u(y)dH^{n-1}(y)\right)^2+\frac{1}{|B|}   \left|\int_{\partial B} u(y)ydH^{n-1}(y)\right|^2+\|R\|_2^2.  \end{equation} 
and moreover, recalling the relation \eqref{fract}, there holds
\begin{equation}\label{2snorm}[u]^2_{\frac{1+s}{2}} =\lambda^s_1 \frac{1}{|B|}  \left|\int_{\partial B} u(y)ydH^{n-1}(y)\right|^2+  [R]^2_{\frac{1+s}{2}}. \end{equation} 

Since $R$ is orthogonal to the harmonics of degree $0$ and $1$, by the monotonicity of the eigenvalues and by \eqref{eig}, there holds a fractional Poincar\'e type inequality
\begin{equation}\label{poincare} [R]^2_{\frac{1+s}{2}}\geq \lambda_2^s \|R\|_2^2=  \frac{2n}{n-s}\lambda^s_1\|R\|_2^2= \frac{2n}{n-s}s H_B^s \|R\|_2^2.\end{equation} 
Therefore 
 \begin{equation}\label{dis2 } \lambda _1^s\|R\|_2^2\leq \frac{n-s}{2n}[R]^2_{\frac{1+s}{2}}\leq \frac{1}{2 }  [R]^{2}_{\frac{1+s}{2}}.  \end{equation}

We rewrite the $H^{\frac{1+s}{2}}$ norm of $u$ as follows:
\begin{eqnarray}\nonumber \|u\|_2^2+ [u]^2_{\frac{1+s}{2}} &=&\frac{ 1}{|B|} \left| \int_{\partial B} u(y)dH^{n-1}(y)\right|^2+\frac{1+\lambda_1^s}{|B|}   \left|\int_{\partial B} yu(y) dH^{n-1}(y)\right|^2+\|R\|_2^2+  [R]^2_{\frac{1+s}{2}}\\ &\leq & \frac{ 1}{|B|} \left| \int_{\partial B} u(y)dH^{n-1}(y)\right|^2+\frac{1+\lambda_1^s}{|B|}   \left|\int_{\partial B} yu(y) dH^{n-1}(y)\right|^2+  \left(1+\frac{1}{2\lambda_1^s}\right) [R]^2_{\frac{1+s}{2}}.\label{norm} 
\end{eqnarray} 

We are going  to estimate each term appearing on the left hand side. 

First of all we observe that by exploiting the barycenter condition $\int_{E} x_i=0$, rewriting the integral in polar coordinates we get  for all $i=1, \dots, n$, $0=\int_{E} x_i =\int_{\partial B} \frac{y_i(1+u(y))^{n+1}}{n+1}dH^{n-1}(y)$.  Now, using a Taylor expansion and recalling that $\int_{\partial B} y_i dH^{n-1}(y)=0$ we get
\[\left|\int_{\partial B} y_i u(y) dH^{n-1}(y)\right|= n \int_{\partial B} y_iu^2(y)(1+\eps O(1)) dH^{n-1}(y)\leq n \|u\|_2^2 (1+\eps O(1))\] from which we deduce 
\begin{equation}\label{dis2}  \left|\int_{\partial B} yu(y)dH^{n-1}(y)\right|^2\leq n \|u\|_2^4(1+\eps O(1)) =\eps \|u\|_2^2 O(1).\end{equation}

Now, by \eqref{2norm} and \eqref{2snorm}  and by \eqref{dis2}, we get 
\begin{equation}\label{auto1} [u]^2_{\frac{1+s}{2}}-\lambda^s_1 \|u\|_2^2+ \lambda_1^s\frac{ 1}{|B|} \left| \int_{\partial B} u(y)dH^{n-1}(y)\right|^2=[R]^2_{\frac{1+s}{2}}-\lambda_1^s\|R\|^2_2\geq \frac{1}{2}[R]^2_{\frac{1+s}{2}}. \end{equation} 
Recalling that $\lambda_1^s=sH_B^s$ and using \eqref{secondoint} to substitute  $[u]^2_{\frac{1+s}{2}}-\lambda^s_1 \|u\|_2^2 $ in the previous inequality we get 
\[ \int_{\partial B} u(x)(H^s_E(x(1+u(x))-H_B^s)dH^{n-1}(x)  +\frac{ sH_B^s}{|B|} \left| \int_{\partial B} u(y)dH^{n-1}(y)\right|^2  \geq  \frac{1}{2}[R]^2_{\frac{1+s}{2}}\]
from which, by H\"older inequality, we 
conclude that 
\begin{equation}\label{auto}  \frac{1}{2}[R]^2_{\frac{1+s}{2}}\leq    \frac{ sH_B^s}{|B|} \left| \int_{\partial B} u(y)dH^{n-1}(y)\right|^2 +\|u\|_2\|H^s_E-H^s_B\|_{L^2(\partial B)}. \end{equation} 

Observe that by \eqref{primoint} and H\"older inequality we get 
\begin{multline*} \sqrt{n|B|} \|H_E^s-H^s_B\|_{L^2(\partial B)}\geq \left|\int_{\partial B}( H^s_E(x(1+u(x))-H^s_B)d H^{n-1}(x)\right| \\ \geq \frac{sH_B^s}{2}\left|\int_{\partial B}  u(x)   dH^{n-1}(x)\right| -  \frac{n+s}{2}[u]^2_{\frac{1+s}{2}}(1+\eps O(1))-  sH_B^s\|u\|_2^2 (1+\eps O(1)). \end{multline*}
In particular this implies that for some constant $C(n,s)>0$ depending on $n,s$ there holds 
\begin{equation}\label{dis1}  \left|\int_{\partial B}  u(x)   dH^{n-1}(x)\right|^2\leq C(n,s)\left[ \|H_E^s-H^s_B\|_{L^2(\partial B)}^2+ ( [u]^2_{\frac{1+s}{2}}+\|u\|_2^2)^2(1+\eps O(1))\right].
\end{equation}

By using \eqref{poincare}, \eqref{dis2}, \eqref{auto}, \eqref{dis1},   we get that there exists a constant $C(n,s)>0$ such that 
\[ \|u\|_2^2+ [u]^2_{\frac{1+s}{2}} \leq C(n,s)\left[  \|H_E^s-H^s_B\|_{L^2(\partial B)}^2+ \|H_E^s-H^s_B\|_{L^2(\partial B)}\|u\|_2\right]+\eps O(1)(\|u\|_2^2+ [u]^2_{\frac{1+s}{2}} ).\] 
By Young inequality, we conclude \eqref{deficit}.

\vspace{0.3cm} 

\noindent
{\bf Second step: proof of \eqref{deficit2}}. 

First of all note that by area formula and by  the estimate \eqref{terzoint} we get that
\begin{eqnarray*}
\overline{H^s_E}&=& \frac{1}{\Per(E)} \int_{\partial B} H^s_E(x(1+u(x)x)\sqrt{(1+u(x))^2+|\nabla u(x)|^2}(1+u(x))^{n-2}dH^{n-1}(x)\\ 
&=&   \frac{\Per(B)}{\Per(E)} H^s_B (1+\eps O(1)) -\frac{n+s}{2\Per(E)}([u]^2_{\frac{1+s}{2}}-sH_B^s \|u\|^2_2)(1+\eps O(1))\\
&=& H_B^s +\frac{\Per(E)-\Per (B)}{\Per(E)} H^s_B  -\frac{n+s}{2\Per(E)}([u]^2_{\frac{1+s}{2}}-sH_B^s \|u\|^2_2)(1+\eps O(1)). \end{eqnarray*} 
By the area formula and a linearization argument (see \cite{mps}) we get that $0\leq \Per(E)-\Per(B)\leq C(n) \|u\|^2_{H^1(\partial B)}$.  
Therefore we conclude that  there exists $\lambda\in \R$ with $|\lambda|\leq C(n,s)\eps$ for some constant $C(n,s)$ only depending on $n,s$ such that 
\[ \overline{H^s_E}= H^s_B(1+ \lambda). \]
We define $E_\lambda=(1+\lambda)^{\frac{1}{s}} E$, that is  $E_\lambda$ is the nearly spherical set associated with the function $u_\lambda= (1+\lambda)^{\frac{1}{s}} -1+(1+\lambda)^{\frac{1}{s}} u$. Note that $  \overline{H^s_{E_\lambda}}= H^s_B$ and  by \eqref{deficit} applied to $u_\lambda$ we get
\[  [u_\lambda]_{\frac{1+s}{2}}^2+\|u_\lambda\|_2^2 \leq C(n,s) \|H_{E_\lambda}^s-\overline{H_{E_\lambda}^s} \|_{L^2(\partial B)}^2= \frac{
C(n,s)}{(1+\lambda)^2}  \|H_{E}^s-\overline{H_{E^s}} \|_{L^2(\partial B)}^2.\] 
Observing that $[u_\lambda]^{2}_{\frac{1+s}{2}}=(1+\lambda)^{\frac{2}{s}} [u ]^{2}_{\frac{1+s}{2}}$ we get for $\eps$ sufficiently small 
\[[u]_{\frac{1+s}{2}}^2\leq 2C(n,s)\|H_{E}^s-\overline{H_{E^s}} \|_{L^2(\partial B)}^2.   \] 
Finally, we recall the Poincar\'e type inequality \eqref{auto1} 
\[  [u]^2_{\frac{1+s}{2}}\geq \lambda^s_1 \|u\|_2^2- \lambda_1^s\frac{ 1}{|B|} \left| \int_{\partial B} u(y)dH^{n-1}(y)\right|^2\] 
and  the fact that, by the volume condition $|E|=|B|$, there holds (see \eqref{volume})
\[ \left| \int_{\partial B} u(y)dH^{n-1}(y)\right|^2\leq C \|u\|^4_2.\]
Therefore we obtain 
\[ [u]_{\frac{1+s}{2}}^2+\|u\|_2^2 \leq C(n,s) \|H_{E }^s-\overline{H_{E}^s} \|_{L^2(\partial B)}^2.\] 
Finally we observe that, if $\|u\|_{C^1}\leq \eps_0$, there exists a constant depending on the dimension such that 
\begin{equation}\label{tan} 
C(n)^{-1} \|H^s_E- \overline{H_{E^s}}\|_{L^2(\partial E)}\leq \|H^s_E-\overline{H_{E^s}}\|_{L^2(\partial B)}\leq C(n) \|H^s_E-\overline{H_{E^s}}\|_{L^2(\partial E)}. 
\end{equation} 
Using this estimate and the previous inequality, we obtain \eqref{deficit2}.  
\end{proof}


\section{Volume preserving flow of nearly spherical sets} 
 In this section we consider the long long time behavior of the volume preserving mean curvature flow \eqref{mcf} starting from nearly spherical sets.
  In particular 
 we will show that  if $E_0$ is a nearly spherical set sufficiently close to a sphere $B_m$, then the flow $E_t$ exists for all times and converges exponentially fast to the reference sphere $B_m$, eventually translated by some vector $\bar b$,   as $t\to +\infty$.  This result will be obtained by using the short time existence result in \cite{jlm}, the Alexandrov   theorem for the fractional mean curvature proved in  \cite{cfmn}, \cite{cabre} and the quantitative inequality \eqref{losia2} obtained in Theorem \ref{mediacurv}, which can be regarded as a \L ojasiewicz-Simon inequality for the  geometric functional $\Per_s(E)-\Per_s(B_m)$. Similar arguments  have been used in \cite{s}, see also \cite{mp},  to provide full convergence of geometric gradient flows. We refer to \cite{at}, \cite{e} (see also \cite{h}) for an analogous result in the local case.
 
 First of all we observe that we may restrict without loss of generality to the case in which the reference ball has volume  $m=1$. Indeed the general case can be obtained via a simple rescaling argument.

 We start recalling the fractional analogue  of the classical Alexandrov theorem:
 
\begin{theorem}[\cite{cabre},\cite{cfmn}] \label{alek} If $\Omega$ is a bounded open set with boundary of class $C^{1, s}$ and $H_{\Omega}^s$ is constant on $\partial \Omega$, then 
$\partial \Omega$  is
a sphere.
 \end{theorem}

Short time existence of a smooth solution to \eqref{mcf} has been proved in \cite{jlm} for  compact initial data with $C^{1,1}$ boundary, by parametrizing the flow   using the height function over a smooth reference surface and by exploiting a fixed point argument. 

\begin{theorem}[\cite{jlm}] \label{ex}  Let $\alpha\in (0, 1-s)$, and $\Sigma$ be a  smooth compact surface and assume that the initial datum  $\partial E_0$ can be written as the graph of a function $u_0$ on $\Sigma$, which is called the height function.

Then there exists $\delta_0>0$, $T_0>0$ and $C_k>0$ for $k\geq 2$  such that if  $\|u_0\|_{C^{1, s+\alpha}( \Sigma)}\leq  \delta_0$   then    \eqref{mcf} has a 
  unique classical solution $E_t$ for $t\in [0, T_0)$ starting from $E_0$,  moreover $\partial E_t$ is the graph of a smooth function $u(x,t)$ on the surface $\Sigma$ which satisfies
  \[\sup_{t\in [0, T_0)}\|u\|_{C^{1,\alpha+s}} \leq 2 \delta_0\]
  and for every $k>1$
  \[\sup_{t\in [0, T_0)} t^{k!} \|u\|_{C^k}\leq C_k.  \]

  \end{theorem}  
Since for every compact set $E_0$ with  $C^{1,1}$ boundary there exists a reference smooth surface $\Sigma$ such that $\partial E_0$ can be written as a   graph on    $\Sigma$ of a function with  $\|u_0\|_{C^{1, s+\alpha}( \Sigma)}\leq \delta$  and $\|u\|_{C^0(\Sigma)} \leq \eps (\delta)<\delta$ the theorem implies short time existence of smooth solutions to \eqref{mcf} with compact $C^{1,1}$ initial datum.

We restrict now the class of initial data of the flow to nearly spherical sets, in which therefore the reference surface is   given by $\partial B$. 
We parametrize the flow in terms of the height function on the reference surfaces $\partial B$, see \cite{SAEZ} and \cite{jlm}. 
Let $E_0$ be a nearly spherical set according to Definition \ref{defi1}. Then the external normal of $E_0$ at a point $p=(1+u_0(x))x$, with $x\in\partial B$ can be expressed as
\begin{equation}\label{normal}
\nu(p)=\frac{(1+u_0(x))x-\nabla u_0(x)  }{\sqrt{(1+u_0(x))^2+|\nabla u_0(x)|^2}}\end{equation}
and as long as $E_t=\{p=rx \ : \ x\in\partial B,  r\in [0, 1+u(x,t)]\}$, then $u(x,t)$ satisfies 
\begin{equation}\label{eq}  \begin{cases} u_t(x,t)=-\left[H^s_{E_t}(x(1+u(x,t))-\overline{ H^s_{E_t}} \right] \frac{\sqrt{(1+u(x,t))^2+|\nabla u(x,t)|^2}}{1+u(x,t)} & t\in (0, T) \\
u(x,0)=u_0(x). & \end{cases} 
\end{equation}   Also viceversa, if $u$ is a solution to \eqref{eq} in $[0, T)$, then the set defined as $E_t=\{p=rx \ : \ x\in\partial B,  r\in [0, 1+u(x,t)]\}$  is a solution to the flow \eqref{mcf} in $[0, T)$.

We state now our main result, which gives  that  the volume preserving flow starting from a  set which is sufficiently close to the ball in $C^1$ norm smoothly converges to the ball  itself (eventually translated by a fixed vector).  

\begin{theorem}\label{convergenza1}  Let  $E_0$ be a nearly spherical set on a given ball $B$, according to Definition \ref{defi1}, with $|E_0|=|B|$. Let $C=\|u_0\|_{C^{1,1}(\partial B)}$, then there exists $\eps=\eps(C)>0$ such that if  $\|u_0\|_{C^1}<\eps $, then 
the flow $E_t$ of \eqref{mcf} starting from $E_0$ exists smooth for every time $t$ and moreover   $E_t-\bar b\to B$ in $C^\infty$, for some $\bar b\in \R^n$. Moreover, there exists a constant $C(n,s)$  depending on $n,s$ such that   
\[ \Per_s(E_t)-\Per_s(B)\leq   C(n,s)(\Per_s(E_0)-\Per_s(B)) e^{-C(n,s)t} \qquad \forall t\geq 0\]
and for all $m\geq 1$ there exists a constant $C(m,n,s)>0$ 
 \[\|u(x,t)-(\bar b\cdot x)x\|_{C^m(\partial B)} \leq C(m,n,s)(\Per_s(E_0)-\Per_s(B)) e^{-C(m,n,s)t}\qquad \forall t\geq 0.\]
\end{theorem} 
\begin{proof} Along the proof, $C(n)$ will indicate a dimensional constant which may change from line to line, and $C(n,s)$ will indicate a constant depending on $n,s$ which may change from line to line. 

By interpolation inequality, we get that  for $\alpha\in (0, 1-s)$, $\|u_0\|_{C^{1, s+\alpha}}\leq C'\eps^{1-s-\alpha} $, where $C'>0$ depends on $C, s, \alpha$. So, if $\eps$ is sufficiently small such that $C'\eps^{1-s-\alpha}\leq \delta_0$, where $\delta_0$ is  as in  Theorem \ref{ex}, then  the flow $E_t$ exists smooth for $t\in [0, T_0)$, and we have that $E_t$ is a nearly spherical set on $\partial B$ with height function $u(\cdot, t)$ with  $\sup_{t\in [0, T_0)}\|u\|_{C^{1+\alpha+s}} \leq  2 \delta_0$ and for every $k>1$, $\sup_{t\in [0, T_0)} t^{k!} \|u\|_{C^k}\leq C_k $.

Note that in particular \begin{equation}\label{delta} \sup_{t\in [0, T_0)} \|u(\cdot, t)\|_{C^1}\leq 2\delta_0.\end{equation}  Moreover $u$ is a solution to \eqref{eq} in $[0, T_0)$ and in particular, recalling also \eqref{tan}, we get 
\begin{eqnarray}\label{derivatanormale} \|u_t\|^2_{L^2(\partial B)}&=&\int_{\partial B} u_t^2(x,t) = \int_{\partial B}  \left[H^s_{E_t}(x(1+u(x,t))-\overline{ H^s_{E_t}} \right]^2 \left(1+\frac{|\nabla u(x,t)|^2}{ (1+u(x,t))^2}\right)\\ &\leq& \|H^s_{E_t} -\overline{ H^s_{E_t}} \|^2_{L^2(\partial B)}(1+\delta_0 O(1)) \leq  C(n,s)  \|H^s_{E_t} -\overline{ H^s_{E_t}} \|^2_{L^2(\partial E_t)} . \nonumber
\end{eqnarray} 

We recall the evolution law of some geometric quantities associated with the flow \eqref{mcf} (see \cite[Section 2]{cinti}). 
First of all it is easy to check that the flow preserves the volume of the set $E_t$ since $ \frac{d}{dt} |E_t|=-\int_{\partial E_t} (H^s_{E_t}(y)-\overline{H^s_{E_t}})dH^{n-1}(y)=0$, moreover  
\begin{equation}\label{derivata}   \frac{d}{dt} \Per_s(E_t)=-\int_{\partial E_t}( H^s_{E_t}(y)-\overline{H^s_{E_t}})H^s_{E_t}(y)dH^{n-1}(y)=-\|H_{E_t}^s- \overline{H^s_{E_t}}\|_{L^2(\partial E_t)}^2. \end{equation} 
In particular this implies that $\Per_s(E_t)\leq \Per_s(E_0)$, for all $t\in (0, T_0)$. 

We compute the barycenter of $E_t$, by using polar coordinates:
\[ b_{E_t}= \frac{1}{|E_t|} \int_{E_t}  y dy= \frac{1}{|B|(n+1)} \int_{\partial B}  x (1+u(x,t))^{n+1} dH^{n-1}(x). 
\]
Assuming that  $\delta$ is sufficiently small, and recalling \eqref{delta} we may perform a Taylor expansion getting
\[ b_{E_t}= \frac{1}{|B|} \int_{\partial B}  x \left[u(x,t)+  \frac{n}{2} u^2(x,t)(1+\delta O(1))\right]dH^{n-1}(x)
\] where $O(1)$ is a  generic bounded function on $\partial B$, 
so that $|b_{E_t}|\leq C(n) \|u\|_{L^2(\partial B)} $ for some dimensional constant $C(n)$. 
In particular this implies that $E_t-b_t$ is a nearly spherical set on $\partial B$ with height function $\tilde u(x,t):=u(x,t)-( b_{E_t}\cdot x)x$ which still satisfies \[\sup_{t\in [0, T_0)}\|\tilde u\|_{C^1}\leq C(n)\delta_0 \]   for some $C(n)>0$ dimensional constant. 
This implies that, eventually choosing a smaller $\delta_0$, we may apply the quantitative Alexandrov inequality \eqref{losia2} obtained in Theorem \ref{mediacurv} to the set $E_t-b_{E_t}$.   Recalling that the fractional perimeter and the fractional mean curvatures are independent by spatial translations,  this inequality reads as follows:
\begin{equation}\label{losiano} \Per_s(E_t )-\Per_s(B)\leq K(n,s) \| H_{E_t}^s-\overline{H_{E_t}^s} \|_{L^2(\partial E_t)}^ 2. \end{equation}

We define the function \[H(t)=\left[\Per_s(E_t)-\Per_s(B)\right]^{\frac{1}{2}}. \] Obvioulsy, by \eqref{derivata}, $H(t)$ is decreasing in time and $H(t)\leq \sqrt{\Per_s(E_0)-\Per_s(B)}$. 
Using \eqref{derivata} and \eqref{losiano}, and  recalling \eqref{derivatanormale}, we get the following
\begin{eqnarray}\label{stima} 
\frac{d}{dt} H(t) &=&-\frac{1}{2} \frac{1}{H(t)} \|H_{E_t}^s- \overline{H^s_{E_t}}\|_{L^2(\partial E_t)}^2
\leq -\frac{1}{2\sqrt{K(n,s)}} \|H_{E_t}^s- \overline{H^s_{E_t}}\|_{L^2(\partial E_t)}\\ 
&\leq&-\frac{1}{2\sqrt{K(n,s)C(n,s)}} \|u_t\|_{L^2(\partial B)}.  \nonumber \end{eqnarray} 

Let us fix   $0\leq t_1<t_2<T_0$, and integrate \eqref{stima} between $t_1$ and $t_2$. For some constant $C(n,s)>0$ depending on $n,s$, we have
that 
\begin{equation} \label{disL2}H(0)\geq H(t_1)-H(t_2)\geq C(n,s)\int_{t_1}^{t_2} \|u_t\|_{L^2(\partial B)} \geq C(n,s)  \|u(\cdot, t_2)-u(\cdot, t_1)\|_{L^2(\partial B)}\end{equation} 
where we used that for a smooth function $f:\partial B\times[t_1, t_2]\to \R$ there holds $\int_{t_1}^{t_2} \|f\|_{L^2(\partial B)}\geq \left\|\int_{t_1}^{t_2} f\right\|_{L^2(\partial B)}$, see e.g. \cite[proof of Theorem 1.2]{mp}. 

We recall now the Fuglede type inequality \eqref{dino}  provided in  \cite[Theorem 6.2]{dinoruva} which gives
\[H(0)= \sqrt{\Per_s(E_0)-\Per_s(B)}\leq  \sqrt{K(n,s)  [u_0]^2_{\frac{1+s}{2}} } \leq C(n,s)\eps.  \]   
So, \eqref{disL2} implies for all $t< T_0$, 
\[\|u(\cdot, t)\|_{L^2(\partial B)}\leq  C(n,s)  \|u(\cdot, t)-u_0(\cdot)\|_{L^2(\partial B)}+    \|u_0\|_{L^2(\partial B)}\leq (C(n,s)+1)\eps.  \]

By Gagliardo-Nirenberg-Sobolev inequality,    for $k> m\geq 2$  we get 
\[\|u(\cdot, t)\|_{H^m(\partial B)}\leq C(n) \|u(\cdot, t)\|_{L^2(\partial B)}^{1-\frac{m}{k}} \|u(\cdot, t)\|_{H^k(\partial B)}^{\frac{m}{k}}\] for all $t\in (T_0/2, T_0)$,
which in turns, by the previous estimate and by Theorem \ref{ex},   since $\|u(\cdot, t)\|_{H^k(\partial B)}\leq C(n) \|u(\cdot, t)\|_{C^k(\partial B)}$, 
\[\|u(\cdot, t)\|_{H^m(\partial B)}\leq C(n,s) T_0^{-m(k-1)!} C_k^{\frac{m}{k}}\eps^{1-\frac{m}{k}}.\]
By Sobolev embedding, taking $m$ sufficiently large, we conclude that
\[\|u(\cdot, t)\|_{C^{1, s+\alpha}}\leq C(n,s) T_0^{-m(k-1)!} C_k^{\frac{m}{k}}\eps^{1-\frac{m}{k}}.\]

Observe that we may choose $\eps>0$ sufficiently small in order to have that $\|u(\cdot, t)\|_{C^{1, s+\alpha}}<\delta_0$ and so, we may apply again Theorem \ref{ex}, to extend the solution on a time interval $[0, 2T_0)$. By iterating the argument, we conclude that the solution exists smooth for all $t\geq 0$. 

Note that by using  \eqref{losiano} and \eqref{stima}, we have also that for all $t\geq 0$
\[\frac{d}{dt} H(t)\leq -\frac{C(n,s)}{2K(n,s)} H(t),\] which implies that $H(t)\leq H(0)e^{-C(n,s)t}$ and so in particular
\[0\leq \Per_s(E_t)-\Per_s(B)\leq (\Per_s(E_0)-\Per_s(B)) e^{-2C(n,s)t} .\]
Moreover, the estimate on $H(t)$ implies, through \eqref{disL2}, that $ u(\cdot, t)$ is a Cauchy sequence in $L^2(\partial B)$  as $t\to +\infty$, 
that is  for all $t_2>t_1$
\[  \|u(\cdot, t_2)-u(\cdot, t_1)\|_{L^2(\partial B)}\leq C(n,s) H(0)e^{-C(n,s)t_1}\]
and by the same argument as before based on Gagliardo-Nirenberg-Sobolev inequality, and Sobolev embedding, it is also a Cauchy sequence in $C^m(\partial B)$ as $t\to +\infty$ for all $m\geq 1$. This implies that $u$ converges to some limit function $\bar u:\partial B\to \R$ as $t\to +\infty$ in $C^m(\partial B)$.  
Therefore,  $\bar E=\{ rx,\ r\in [0, 1+\bar u(x)], x\in \partial B\} $ is a regular set which solves $H^s_{\bar E}(y)=\overline{H^s_{\bar E}}$ for all $y\in \partial \bar E$. 
So, by Theorem \ref{alek}  we conclude that $\bar E= B+\bar b$ for some $\bar b\in \R^n$.  
\end{proof}


We conclude observing that the previous argument gives also  an improvement of a result on the  long time behavior of the flow \eqref{mcf} obtained  in \cite{cinti} for convex initial sets $E_0$ under the assumption that the flow exists smooth for all times.  More precisely in \cite{cinti} it is assumed the following regularity condition: 
\begin{assumption}\label{ass1} 
 If $H^s_{E_t}$ is bounded on $E_t$  for all $t\in [0, T_0)$ for some $T_0\leq T$, where $T$ is the maximal time of existence of the flow \eqref{mcf}, then 
 the $C^{2,\beta}$ norm  of $\partial E_t$, up to translations, is also bounded for some $\beta>s$ by a constant only depending on the supremum of $H_s$. In addition either $T=T_0=+\infty$ or $T_0<T$.\end{assumption} 
A particular case in which assumption \eqref{ass1} has been proven to hold for the fractional mean curvature flow  \eqref{mcfbis} 
 is the case the initial set is the subgraph of a Lipschitz continuous function and has bounded fractional curvature, see \cite{cn}.  
 
We recall the result in \cite{cinti}. 
\begin{theorem}[\cite{cinti}]  \label{teocinti} 
Let $E_0$ be a smooth compact convex set. Let $E_t$ be a solution to \eqref{mcf} in $[0, T)$, where $T$ is the maximal time of existence,  and assume that \eqref{ass1} holds. 
Then the flow $E_t$ is defined for all times $t\in [0, +\infty)$,  $E_t$ is smooth and convex, and there exist $b_t\in \R^n$ such that $E_t- b_t$  converges in $C^2$ as $t\to +\infty$  to  a ball with volume $|E_0|$.  \end{theorem} 
Our argument provide a refinement of the previous result, ruling out the translations in time: 
\begin{corollary}\label{corocinti} Under the assumption of Theorem \ref{teocinti}, then  $E_t$  converges exponentially fast in $C^\infty$ as $t\to +\infty$  to  a ball with volume $|E_0|$.  \end{corollary} 
\begin{proof}   Without loss of generality we assume that $|E_0|=|B|$.  By Theorem \ref{teocinti}, we have that 
  for $\eps>0$, there exists $t_\eps$, $b_{t_\eps}$ such that  $E_{t_\eps}- b_{t_\eps}$ is a $C^2$ set  with 
  $\sup_{x\in( ( E_{t_\eps}- b_{t_\eps})\Delta B} d(x, \partial B)\leq C\eps$ and $|\nu_{E_{t_\eps}- b_{t_\eps}}(y)-y|\leq C\eps$ for all $y\in \partial(E_{t_\eps}- b_{t_\eps})$. 
  Then $E_{t_\eps}-b_{t_\eps}$ can be written as a nearly spherical set on $B$, with height function $u_\eps$ with $\|u_\eps\|_{C^1}\leq C\eps$.  We apply now Theorem \ref{convergenza1} to the flow starting from $E_{t_\eps}- b_{t_\eps}$ and we conclude that if we choose $\eps>0$ sufficiently small, we obtain that    $E_t-b_ {t_\eps}-\bar b\to B$ in $C^\infty$, as $t\to +\infty$, for some $\bar b\in \R^n$ with exponential rate of convergence. 
   \end{proof} 
   
\section{Evolution of  periodic graphs} 

In this last section we show that  similar arguments as for nearly spherical sets can be used also to provide the exponential convergence to an hyperplane 
of entire periodic graphs in $\R^n$, evolving by  the fractional mean curvature flow \eqref{mcfbis}. 
More precisely we consider the  geometric flow \eqref{mcfbis}  under the additional assumption that the boundary of the initial  datum $E_0$ can be written as a periodic entire graph on an hyperplane, that is, there exists  $e\in \R^{n}$, such that $\nu(x)\cdot e>0$  for every $x\in \partial E_0$. By monotonicity of the flow it is possible  to show that the evolution $E_t$ maintains  this property for all positive times
$t>0$, that is,   $\nu(x_t)\cdot e>0$  for every $x_t\in \partial E_t$. 

Note that, without loss of generality,  we  may  assume that $e=e_{n}$, up to a rotation of coordinates. 
 
So, let us consider a set $E_0$ which is given by the subgraph of a Lipschitz continuous,   periodic function $u_0$.  Without loss of generality, we will assume that the periodicity cell of $u_0$ is $[0,1]^{n-1}$, so $u_0$ is $\Z^{n-1}$ periodic. 

It is standard to show (see \cite{cn}) that if $E_0=\{(x,z)\ : z\leq u_0(x), x\in \R^{n-1} \}$, then the solution $E_t$  of \eqref{mcfbis} coincides with $\{(x,z)\ : z\leq u(x,t), x\in \R^{n-1} \}$,
where $u$ solves the following 
nonlocal quasilinear system:
\begin{equation}\label{levelset}\begin{cases}  u_t(x,t)=&- H^s_{E_t}(x, u(x,t))\sqrt{1+|\nabla u(x,t) |^2} \ \\
u(x,0)=&u_0(x).\end{cases}\end{equation}
Moreover if the initial datum $u_0$ is Lipschitz continuous, and $\Z^{n-1}$ periodic, then by comparison arguments we get that $u(\cdot, t)$ is Lipschitz continuous (with Lipschitz norm bounded by the Lipschitz norm of $u_0$) and $\Z^{n-1}$ periodic. 
 
We recall a result about existence of a smooth solution to \eqref{mcfbis} and long time behavior obtained in \cite[Theorem 3.3, Proposition 6.1]{cn}.

\begin{theorem}[\cite{cn}] \label{teografo}  Assume that $E_0=\{(x,z)\ : z\leq u_0(x), x\in \R^{n-1} \}$, and that the function $u_0$ is $\Z^n$ periodic and satisfies  $u_0\in C^{1, s+\alpha}(\R^{n-1})$   with  $\|u\|_{C^{1, s+\alpha}(\R^{n-1})} \leq C$ for some $\alpha>0$. Then the flow  $E_t=\{(x,z)\ : z\leq u(x,t), x\in \R^{n-1} \}$  of \eqref{mcfbis} starting from $E_0$ exists smooth for every time $t>0$, with  norms bounded in $\R^{n-1}\times [t_0, +\infty)$  by constants depending on $t_0$, $C$. Moreover  there exists $c\in \R$ such that $u(x,t)\to c$ as $t\to +\infty$ uniformly in $C^1(\R^{n-1})$. 
 \end{theorem} 

In order to improve this convergence result to exponential convergence in $C^\infty$, we need first of all to derive an analog of the  inequality \eqref{losia2} obtained in Theorem \eqref{mediacurv} for  a $C^{1,1}$ function $u_0:R^{n-1}\to \R$, which is $\Z^{n-1}$ periodic and satisfies $\|u_0\|_{C^1}\leq \eps$, for $\eps>0$ sufficiently small. 

For a set $E$ which is given by the subgraph of a periodic function $u$,  we define the periodic fractional perimeter as follows, by  considering the localized version of  \eqref{ps} on the set $[0,1]^{n-1}\times \R$:
\begin{eqnarray}\nonumber  \Per_s^p(E)& :=& \int_{E\cap( [0,1]^{n-1}\times \R)}\int_{\R^n\setminus E} \frac{1}{|x-y|^{n+s}}dydx \\\label{perper} 
&=&\int_{[0,1]^{n-1}}\int_{\R^{n-1}}\int_{u(y)}^{+\infty} \int_{-\infty}^{u(x)}  \frac{ 1}{(|x-y|^2+(w-z)^2)^{\frac{n+s}{2}}} dwdzdydx. 
\end{eqnarray}  
We  recall that $\partial E=\{(x, u(x)), \ :\ x\in\R^{n-1}\}$, and for $p \in \partial E$, where $p=(x, u(x))$ for some $x\in\R^{n-1}$,  the exterior normal to $E$ is given by $\nu(p)= \frac{(-\nabla u(x), 1)}{\sqrt{1+|\nabla u(x)|^2}}$. 
So, for $p=(x, u(x))\in \partial E$ we have that $H^s_E(x, u(x))$ as defined in \eqref{hs} coincides with   
\begin{eqnarray*} 
H_E^s(x,u(x)) &=& \frac{2}{s} \int_{\R^{n-1}} \frac{(y-x, u(y)-u(x))\cdot (-\nabla u(y),1) }{(|y-x|^2+ (u(x)-u(y))^2)^{\frac{n+s}{2}}} dy\\ 
  &=& \frac{2}{s} \int_{\R^{n-1}} \frac{ u(y)-u(x) +(x-y)\cdot \nabla u(y)}{(|y-x|^2+(u(x)-u(y))^2)^{\frac{n+s}{2}}} dy .
  \end{eqnarray*}

We introduce  the (squared) fractional Gagliardo seminorm of $u$ (recalling that the periodicity cell of $u$ is $[0,1]^{n-1}$) which is defined as
\begin{equation}\label{gsgraph}[u]^2_{\frac{1+s}{2}}:=\int_{[0,1]^{n-1}} \int_{\R^{n-1}} \frac{(u(x)-u(y))^2}{|x-y|^{n+s}} dxdy.  \end{equation}
Moreover we will indicate with $\|u\|_2^2$ the squared $L^2$ norm of $u$ on its periodicity cell, that is $\int_{[0,1]^{n-1}} u^2(x)dx$. 
We recall the following Poincar\`e type inequality, see \cite{brezis}: 
there exists a dimensional constant such that if $u$ is a periodic function with $\int_{[0,1]^{n-1}} u(x)dx=0$, there holds 
\begin{equation}\label{frazpoincare} 
[u]^2_{\frac{1+s}{2}}\geq \frac{C(n)}{1-s} \|u\|_2^2. 
\end{equation}   
We prove now a rigidity type result in the same spirit of \eqref{dino}.
\begin{lemma}\label{lemmadino} Assume that $E$ is the subgraph $E$ of a periodic function $u\in C^{1,1}(\R^{n-1})$. Let $H_c$ be the hyperplane $\{(x,z)\ x\in \R^{n-1}, z\leq c\}$. 
Then there holds
\[\frac{1}{2(1+4\|\nabla u\|^2_\infty)^{\frac{n+s}{2}}} [u]^2_{\frac{1+s}{2}}\leq  \Per_s^p(E)- \Per_s^p(H_c)\leq \frac{1}{2} [u]^2_{\frac{1+s}{2}}. \]
\end{lemma} 
\begin{proof}Recalling the definition \eqref{perper} and using the periodicity of $u$, we observe that 
\begin{eqnarray*} 
\Per_s^p(E) -\Per_s^p(H_c)  &=&\frac12 \int_{[0,1]^{n-1}}\int_{\R^{n-1}}\int_{u(y)}^{+\infty} \int_{-\infty}^{u(x)} \frac{ 1}{(|x-y|^2+(w-z)^2)^{\frac{n+s}{2}}} dwdzdydx
\\ &+&\frac12  \int_{[0,1]^{n-1}}\int_{\R^{n-1}}\int_{u(x)}^{+\infty} \int_{-\infty}^{u(y)} \frac{ 1}{(|x-y|^2+(w-z)^2)^{\frac{n+s}{2}}} dwdzdydx
 \\ &-& \int_{[0,1]^{n-1}}\int_{\R^{n-1}}\int_{c}^{+\infty} \int_{-\infty}^{c} \frac{ 1}{(|x-y|^2+(w-z)^2)^{\frac{n+s}{2}}} dwdzdydx . 
 \end{eqnarray*}
 Observe that 
 \begin{eqnarray*} \int_{-\infty}^a\int_{b}^{+\infty}+\int_{-\infty}^b\int_{a}^{+\infty}= \int_{-\infty}^b\int_{b}^{+\infty}+ \int_{b}^a\int_{b}^{+\infty}+ \int_{-\infty}^a\int_{a}^{+\infty}+ \int_{a}^b\int_{a}^{+\infty}\\ =  \int_{-\infty}^b\int_{b}^{+\infty}+ \int_{-\infty}^a\int_{a}^{+\infty} +\int_{a}^b\int_a^b. \end{eqnarray*}
 So, substituting in the previous formula with $u(x)=a$, $u(y)=b$, we get 
\[\Per_s^p(E) -\Per_s^p(H_c) =  \frac{1}{2}  \int_{[0,1]^{n-1}}\int_{\R^{n-1}}\int_{u(x)}^{u(y)} \int_{u(x)}^{u(y)} \frac{ 1}{(|x-y|^2+(w-z)^2)^{\frac{n+s}{2}}} dwdzdydx.\] 
So, it is immediate to check that 
  \begin{eqnarray*}\Per_s^p(E) -\Per_s^p(H_c) &\leq  & \frac{1}{2}  \int_{[0,1]^{n-1}}\int_{\R^{n-1}}\int_{u(x)}^{u(y)} \int_{u(x)}^{u(y)} \frac{ 1}{|x-y|^{n+s}} dwdzdydx\\
  &=& \frac{1}{2} \int_{[0,1]^{n-1}} \int_{\R^{n-1}} \frac{(u(x)-u(y))^2}{|x-y|^{n+s}} dxdy=\frac{1}{2}[u]^2_{\frac{1+s}{2}}.
\end{eqnarray*}
On the other hand, by a simple change of variables 
  \begin{eqnarray*}\Per_s^p(E) -\Per_s^p(H_c) &= & \frac{1}{2}  \int_{[0,1]^{n-1}}\int_{\R^{n-1}}\int_{0}^{\frac{u(y)-u(x)}{|x-y|}}  \int_{0}^{\frac{u(y)-u(x)}{|x-y|}}\frac{ |x-y|^2 }{|x-y|^{n+s} (1+(w-z)^2)^{\frac{n+s}{2}}} dwdzdydx\\
  &\geq &\frac{1}{2(1+4\|\nabla u\|^2_\infty)^{\frac{n+s}{2}}}   \int_{[0,1]^{n-1}}\int_{\R^{n-1}}\int_{0}^{\frac{u(y)-u(x)}{|x-y|}}  \int_{0}^{\frac{u(y)-u(x)}{|x-y|}}\frac{ |x-y|^2 }{|x-y|^{n+s} } dwdzdydx \\
  &=&\frac{1}{2(1+4\|\nabla u\|^2_\infty)^{\frac{n+s}{2}}}  \int_{[0,1]^{n-1}} \int_{\R^{n-1}} \frac{(u(x)-u(y))^2}{|x-y|^{n+s}} dxdy.   
\end{eqnarray*}
  \end{proof} 
We introduce  the fractional Laplacian of order $1+s$ , which can be defined (up to constants depending on $s$ and $n$) as 
\begin{equation}\label{laplgraph}
(-\Delta)^{\frac{1+s}{2}}u(x)=2 \int_{\R^{n-1}} \frac{u(x)-u(y)}{|x-y|^{n+s}}dy. 
\end{equation} 
By periodicity of $u$ there holds $\int_{[0,1]^{n-1}} \int_{\R^{n-1}} \frac{u(x) (u(x)-u(y)) }{|x-y|^{n+s}} dydx
=\int_{\R^{n-1}} \int_{[0,1]^{n-1}}  \frac{u(x) (u(x)-u(y)) }{|x-y|^{n+s}} dydx$, therefore we get 
 \begin{equation}\label{fractgraph} [u]^2_{\frac{1+s}{2}} = \int_{[0,1]^{n-1}} u(x)(-\Delta)^{\frac{1+s}{2}}u(x)dx.  \end{equation}

\begin{lemma} \label{lemma1} Let $E$ be the subgraph  of a periodic function $u\in C^{1,1}(\R^{n-1})$, with  $\|\nabla u\|_{C^0}\le 1$.   
Then  there holds
\begin{equation}\label{integrale} \int_{[0,1]^{n-1}}  u(x) H_E^s(x,u(x)) dx=[u]^2_{\frac{1+s}{2}}(1+O(\|\nabla u\|_{C^0})).\end{equation}

Moreover there exists $\eps_0=\eps_0(n)\in (0,1) $ and $C(n)>0$  such that if   $\|\nabla u\|_{C^0}< \eps_0$, and $\int_{[0,1]^{n-1}}u(x)=0$,  there holds 
\[\|H^s_E\|_{L^2([0,1]^{n-1})}^2\geq C(n) [u]^2_{\frac{1+s}{2}}.\] 
 Finally by Lemma \ref{lemmadino} we also have
 \begin{equation}\label{losianuova}\|H^s_E\|_{L^2([0,1]^{n-1})}^2\geq 2C(n) (\Per_s^p(E)-\Per_s^p(H_c)) \end{equation} 
for every hyperplane   $H_c=\{(x,z)\ x\in \R^{n-1}, z\leq c\}$. 
 \end{lemma} 
\begin{proof} We write the following Taylor expansions, where $  O(\|\nabla u\|_{C^0})$ is any function $f$ such that $|f(x)|\leq C\|\nabla u\|_{C^0}$ for all $x\in \R^{n-1}$: 
 \[ \frac{1}{(|y-x|^2+(u(x)-u(y))^2)^{\frac{n+s}{2}}} =\frac{1}{|x-y|^{n+s}}-\frac{n+s}{|x-y|^{n+s} }\frac{(u(x)-u(y))^2}{|x-y|^2}  (1+O(\|\nabla u\|_{C^0}^2),
\]
and so, substituting in the previous formula for $H^s_E$ we get
\[  H_E^s(x,u(x))  = \frac{2}{s} \int_{\R^{n-1}} \frac{ u(y)-u(x) +(x-y)\cdot \nabla u(y)}{|x-y|^{n+s}}   (1+O(\|\nabla u\|_{C^0}^2) dy.\] 
 Now, we  fix $x\in [0,1]^{n-1}$  and we define 
 $T(y)= \frac{(u(y)-u(x))(y-x)}{|x-y|^{n+s}}$ for $y\in \R^{n-1}$. We observe that, by the same computations as in \eqref{diver}, we have 
\[  \div T(y)= -\frac{(s+1) (u(y)-u(x))}{|x-y|^{n+s}}+\frac{\nabla u(y)\cdot (y-x)}{|x-y|^{n+s}}.\]  It is easy to check that $\int_{\R^{n-1}} \div T(y)dy=\lim_{R\to +\infty} \int_{B(x,R)} \div T(y)dy=0$
and so we get that 
\[ \int_{\R^{n-1}} \frac{  (x-y)\cdot \nabla u(y)}{|x-y|^{n+s}}dy= -(s+1)  \int_{\R^{n-1}} \frac{  u(y)-u(x)}{|x-y|^{n+s}}dy.\]
Therefore, we  conclude that the fractional mean curvature is given by
$H^s_E$ we get
\[  H_E^s(x,u(x))  =-2\int_{\R^{n-1}} \frac{ u(y)-u(x)}{|x-y|^{n+s}} (1+O(\|\nabla u\|_{C^0}^2) dy .\] 
This implies the conclusion \eqref{integrale}, by recalling the periodicity of $u$. 

Now, assume that $u$ has zero average. We apply H\"older inequality to \eqref{integrale} and we get, recalling also the Poincar\`e inequality \eqref{frazpoincare}  and choosing $\eps_0$ sufficiently small 
\begin{eqnarray*} \|u\|_2\|H^s_E\|_2 & \geq&  \int_{[0,1]^{n-1}}  u(x) H_E^s(x,u(x)) dx\\ &\geq & \frac{1}{2} [u]^2_{\frac{1+s}{2}} \geq \frac{1}{4} [u]^2_{\frac{1+s}{4}} +\frac{C(n)}{4(1-s)} \|u\|_2^2 . \end{eqnarray*} 
We fix $\delta=C(n)/2$ and we conclude by Young inequality that
\begin{eqnarray*} \frac{1}{C(n)}\|H^s_E\|_2^2&=&  \frac{1}{2\delta} \|H^s_E\|_2^2\geq   \frac{1}{4} [u]^2_{\frac{1+s}{2}} +\left(\frac{C(n)}{4(1-s)}-\frac{\delta}{2}\right) \|u\|_2^2\\ &=& \frac{1}{4} [u]^2_{\frac{1+s}{2}} + \frac{C(n)s}{4(1-s)}  \|u\|_2^2 
\end{eqnarray*} 
which implies the conclusion. 
\end{proof} 

 We are ready to prove the exponential convergence  which improve Theorem \ref{teografo}.
 \begin{corollary}\label{expografo} Assume that $E_0=\{(x,z)\ : z\leq u_0(x), x\in \R^{n-1} \}$, where  $u_0\in C^{1, s}(\R^{n-1})$ is $\Z^{n-1}$ periodic  with  $\|u\|_{C^{1, s}(\R^{n-1})} \leq C$. Then the flow  $E_t=\{(x,z)\ : z\leq u(x,t), x\in \R^{n-1} \}$  of \eqref{mcfbis} starting from $E_0$ exists smooth for every time $t>0$, and moreover  there exists $\bar c\in \R$
  such that for all $m\geq 1$    there exists a constant $C(m,n,s)>0$ such that 
 \[\|u(x,t)-\bar c\|_{C^m } \leq C(m,n,s)\sqrt{ [u_0]^2_{\frac{1+s}{2}}} e^{- C(m,n,s)t} \qquad \forall t\geq 0.\]
 \end{corollary}
 \begin{proof} By Theorem \ref{teografo}, we know that the solution exists smooth for all times, and moreover $u\to \bar c$ uniformly in $C^1$ as 
 $t\to +\infty$. So, for every $\eps>0$ there exists $t_\eps$ such that $\|\nabla u(\cdot, t ) \|_{C^0}\leq \eps$ for all $t>t_\eps$. 
 Let $m_t=\int_{[0,1]^{n-1}} (u(x, t)-\bar c)dx$.
If $\eps>0$ is sufficiently small, I may apply to $u(x, t)-\bar c-m_t$ the results in Lemma \ref{lemma1} for all $t>t_\eps$. 
 In particular \eqref{losianuova} reads as follows, observing that the fractional mean curvature and the fractional perimeter are independent by translations, 
\[\|H^s_{E_{t}}\|_{L^2([0,1]^{n-1})}^2\geq 2C(n)(\Per_s^p(E_t)-\Per_s^p(H_c)).\]

We proceed as in the proof of Theorem \ref{convergenza1}. 
First of all by using \eqref{levelset} we have that, for $\eps>0$ small, and $t>t_\eps$, 
\[\|u_t\|_{L^2([0,1]^{n-1})}^2=\int_{[0,1]^{n-1}} (H_{E_t}^s(x,u(x,t) ))^2 (1+|\nabla u(x,t)|^2)dx\leq 2\|H_{E_t}^s\|_{L^2([0,1]^{n-1})}^2.  \]
We define $H(t)=\sqrt{\Per_s^p(E_t)-\Per_s^p(H_c)}$,
so that, by using the previous inequalities and denoting $C(n)$ a dimensional constant, which may change from line to line 
\begin{eqnarray*} \frac{d}{dt} H(t)&=& \frac{1}{2 H(t)}\frac{d}{dt}\Per^p_s(E_t)= -\frac{1}{2 H(t)}\int_{\partial E_t} (H^s_{E_t}(x))^2dH^{n-1}(x)\\ &\leq & -\frac{1}{2 H(t)}\|H_{E_t}^s\|_{L^2([0,1]^{n-1})}^2\leq -\frac{\sqrt{ 2C(n)}}{2}\|H_{E_t}^s\|_{L^2([0,1]^{n-1})} \leq - C(n)\|u_t\|_{L^2([0,1]^{n-1})}. \end{eqnarray*}
 Moreover by the same computation we have that $\frac{d}{dt} H(t)\leq -C(n) H(t)$, so $H(t)\leq H(t_\eps) e^{-C(n)(t-t_\eps)}$ for all $t>t_\eps$. 
 
As in the proof of Theorem \ref{convergenza1}, we deduce that $\|u(\cdot, t)-\bar c\|_{L^2([0,1]^{n-1})}$ is a Cauchy sequence as $t\to +\infty$, and moreover, by the estimate in Theorem \ref{teografo} and Sobolev embedding, $\|u(\cdot, t)-\bar c\|_{C^m([0,1]^{n-1})}$ is  a Cauchy sequence as $t\to +\infty$ for all $m\geq 1$, with exponential rate of convergence. This gives the thesis. 
 \end{proof} 
  

\end{document}